\documentclass[reqno, 12pt]{amsart}
\usepackage{newunicodechar}
\newunicodechar{，}{,}

\usepackage[utf8]{inputenc}
\usepackage[english]{babel}
\usepackage{amsmath,amssymb,amsfonts,amsthm,mathtools}
\usepackage{enumitem}
\usepackage{booktabs}
\usepackage{microtype}
\usepackage{hyperref}
\usepackage[nameinlink,capitalise]{cleveref}
\usepackage[margin=1.15in]{geometry}

\hypersetup{colorlinks=true,linkcolor=black,citecolor=black,urlcolor=black}
\numberwithin{equation}{section}

\theoremstyle{plain}
\newtheorem{theorem}{Theorem}[section]
\newtheorem{lemma}[theorem]{Lemma}
\newtheorem{proposition}[theorem]{Proposition}
\newtheorem{corollary}[theorem]{Corollary}

\theoremstyle{definition}
\newtheorem{definition}[theorem]{Definition}
\newtheorem{example}[theorem]{Example}

\theoremstyle{remark}
\newtheorem{remark}[theorem]{Remark}

\makeatletter

\let\original@ps@headings\ps@headings
\let\original@ps@firstpage\ps@firstpage

\def\ps@headings{%
  \original@ps@headings
  \def\@evenhead{\normalfont\scriptsize
    \hfil\leftmark\hfil}%
  \def\@oddhead{\normalfont\scriptsize
    \hfil\rightmark\hfil}%
  \def\@oddfoot{\hfil\thepage\hfil}%
  \let\@evenfoot\@oddfoot
}

\def\ps@firstpage{%
  \original@ps@firstpage
  \let\@oddhead\@empty
  \let\@evenhead\@empty
  \def\@oddfoot{\hfil\thepage\hfil}%
  \let\@evenfoot\@oddfoot
}

\makeatother

\begin{document}

\title[Generalized Mock-Poisson Algebras]{Generalized Mock-Poisson Algebras: Constructions and Structural Properties}

\author[Zafar Normatov]{Zafar Normatov}
\address{Zafar Normatov \newline \indent School of Mathematics, Jilin University, Changchun, 130012, China,\newline \indent
Institute of Mathematics, Uzbekistan Academy of Sciences, Tashkent, 100174, Uzbekistan} \email{z.normatov@inbox.ru}
\date{}

\begin{abstract}

In this paper, we introduce generalized mock-Poisson algebras as a mock analogue of generalized Poisson algebras. We prove that $A_\ast^3$ is an ideal annihilated
by both operations and that the quotient
$A/A_\ast^3$ is a mock-Poisson algebra.
We also show that the dimension seven is the
smallest dimension in which a generalized
mock-Poisson algebra need not be a
mock-Poisson algebra. We give several constructions using derivations, anti-Zinbiel algebras, tensor products, and one-dimensional extensions. We study polarization and depolarization and characterize the corresponding admissible algebras. We prove that every finite-dimensional generalized mock-Poisson algebra over a field of characteristic zero is nilpotent. Finally, we describe representations in terms of semidirect products.
\end{abstract}
\subjclass[2020]{17A30, 17B63, 17A36, 17A60}

\keywords{generalized mock-Poisson algebra; mock-Poisson algebra; anti-associative algebra; anti-Zinbiel algebra; representation; semidirect product; nilpotency; polarization.}

\maketitle


\section{Introduction}

Poisson algebras form an important class of algebraic structures with connections to several areas of mathematics and mathematical physics. More precisely, a Poisson algebra $(A,\cdot,[-,-])$ consists of a commutative associative algebra $(A,\cdot)$ and a Lie algebra $(A,[-,-])$ satisfying the Leibniz identity
\[
[xy,z]=x[y,z]+[x,z]y,
\]
for all $x,y,z\in A$. The compatibility between the associative product and the Lie bracket is a fundamental feature of Poisson algebras and has led to various generalizations and related algebraic structures \cite{ABCF, AOM, Agore, GozeRemm}.

A generalized Poisson algebra is obtained by replacing the Leibniz identity with
\[
[a\cdot b,c]
=
a\cdot[b,c]+[a,c]\cdot b+(a\cdot b)\cdot D(c),
\]
where $D$ is a derivation of both the associative commutative algebra $(A,\cdot)$ and the Lie algebra $(A,[-,-])$. Thus, this identity reduces to
the usual Leibniz identity when $D=0$, and a generalized Poisson algebra becomes
a Poisson algebra. Generalized Poisson algebras have appeared in connection with Jordan superalgebras, contact-type brackets, tensor product constructions and generalized Wronskian algebras \cite{Agore,CantariniKac,Herman}. In \cite{CNOR}, the authors studied several aspects of the structural theory of generalized Poisson algebras such as nilpotency and solvability, various constructions, polarization and depolarization, and low-dimensional classifications.

Recently, the notion of a \(\delta\)-mock-Poisson algebra was introduced as a parameter-dependent mock analogue of a Poisson algebra in \cite{RN}. For \(\delta=1\), one obtains mock-Poisson algebras. Unlike the classical Poisson setting, one operation is anti-commutative anti-associative, while the other is a mock-Lie bracket and the compatibility identity is the anti-Leibniz identity. Therefore, there is a natural question:
\textbf{\emph{Is there a mock analogue of a generalized Poisson algebra?}}
Motivated by this question, we introduce the following definition.
\begin{definition}
A \emph{generalized mock-Poisson algebra} is a quadruple $(A,\ast,\{-,-\},D)$ such that $(A,\ast)$ is anti-commutative and anti-associative, $(A,\{-,-\})$ is a mock-Lie algebra (see the definition in \cite{BF, zusmanovich}), $D$ is a derivation with respect to both operations and all $x,y,z\in A$ satisfy
\begin{equation}\label{eq:gmp-nil-compatibility}
\{x\ast y,z\}=-x\ast\{y,z\}-\{x,z\}\ast y+(x\ast y)\ast D(z).
\end{equation}
\end{definition}
Note that when $D=0$, this definition recovers the mock-Poisson algebra.

We begin by studying the basic properties of generalized mock-Poisson algebras. We prove that $A_\ast^3$ is an ideal annihilated by both operations and that
the quotient $A/A_\ast^3$ is a mock-Poisson algebra. We also show that if $(x\ast y)\ast D(z)\neq 0$ for some $x,y,z\in A$, then $\dim A\geq 7$. An example in dimension seven shows that this bound
is sharp. Moreover, the underlying anti-commutative anti-associative algebra of a seven-dimensional generalized mock-Poisson algebra is unique up to isomorphism whenever $A_\ast^3\neq 0$.

We next give several constructions of generalized mock-Poisson algebras. These include constructions from derivations and anti-Zinbiel algebras as well as tensor products and one-dimensional extensions. We also study polarization and depolarization and obtain a characterization in terms of a single bilinear operation.

Then we study the nilpotency property of generalized mock-Poisson algebras. Although the algebra is nilpotent with respect to each operation separately in the finite-dimensional case, this does not immediately imply nilpotency for products involving both operations. We use the compatibility identity to study mixed products and prove that every finite-dimensional generalized mock-Poisson algebra over a field of characteristic zero is nilpotent. We also study representations of mock-Poisson and generalized mock-Poisson algebras using semidirect products.

The paper is organized as follows.
Section~\ref{sec2} contains the basic structural properties and constructions of generalized mock-Poisson algebras.
Section~\ref{sec3} deals with polarization and depolarization. Section~\ref{sec4} is devoted to nilpotency, while Section~\ref{sec5} studies representations
and semidirect products.

Throughout, all vector spaces and algebras are defined over a field of characteristic zero unless otherwise stated. 

\section{Basic structural properties and constructions}\label{sec2}

Recall that a subalgebra $\mathcal{B}$ of an algebra $(\mathcal{A},\star_1,\star_2)$ with two binary operations is defined as a linear space that is closed under both products i.e. $\mathcal{B}\star_1 \mathcal{B}+\mathcal{B}\star_2\mathcal{B}\subseteq\mathcal{B}$. A subalgebra $\mathcal{I}$ of $\mathcal{A}$ is called an ideal if it satisfies $\mathcal{I}\star_1\mathcal{A}+\mathcal{A}\star_1\mathcal{I}+\mathcal{I}\star_2\mathcal{A}+\mathcal{A}\star_2\mathcal{I}\subseteq\mathcal{I}$.

Let $(A,\ast)$ be an anti-commutative anti-associative algebra. Set
$A_\ast^1=A$ and
\[
A_\ast^n=\sum_{i=1}^{n-1}A_\ast^i\ast A_\ast^{n-i},\qquad n\ge2.
\]
Then $A_\ast^2=\operatorname{span}\{x\ast y\ | \ \forall x,y\in A\}$ and 
$A_\ast^3=\operatorname{span}\{(x\ast y)\ast z\ | \forall x,y,z\in A\}$ by anti-associativity. By \cite{R22}, $A_\ast^4=0$.
Note that $A_\ast^2\ast A_\ast^2\subseteq A_\ast^4$.

\begin{proposition}\label{prop:central-defect}
Let $\mathcal A=(A,\ast,\{-,-\},D)$ be a generalized mock-Poisson algebra. Then
\[
\{A_\ast^3,A\}=0,
\qquad
A_\ast^2\ast\{A,A\}=0,
\]
and $A_\ast^3$ is a $D$-stable ideal annihilated by both binary operations. Moreover,  $\mathcal A/A_\ast^3$ is a mock-Poisson algebra.
\end{proposition}

\begin{proof}
Let us set $L(x,y,z,t):=\{(x\ast y)\ast z,t\}$. Applying \eqref{eq:gmp-nil-compatibility} to $(x\ast y,z,t)$ and then to $\{x\ast y,t\}$ with removing terms in $A_\ast^4$, we obtain
\begin{equation}\label{L}
L(x,y,z,t)=-(x\ast y)\ast\{z,t\}+(x\ast \{y,t\})\ast z+(\{x,t\}\ast y)\ast z.
\end{equation}
Applying \eqref{eq:gmp-nil-compatibility} to $(x,y\ast z,t)$ and then to $\{y\ast z,t\}$, we obtain
\begin{equation}\label{-L}
\begin{split}
-L(x,y,z,t)=&x\ast(y\ast \{z,t\})+x\ast (\{y,t\}\ast z)-\{x,t\}\ast(y\ast z)\\
=&-(x\ast y)\ast \{z,t\}-(x\ast \{y,t\})\ast z+(\{x,t\}\ast y)\ast z.
\end{split}
\end{equation}
By adding and subtracting  \eqref{L} and \eqref{-L}, we obtain
\[
(x\ast y)\ast \{z,t\}=(\{x,t\}\ast y)\ast z, \qquad \text{and} \qquad  L(x,y,z,t)=(x\ast\{y,t\})\ast z.
\]
Let us define $h(x,y,z;t):=(x\ast y)\ast \{z,t\}$. Then, $L=-h(x,z,y;t)$. Moreover, $h$ is antisymmetric in
$x,y,z$  and symmetric in $z,t$. Therefore, 
\[
\begin{split}
h(x,y,z;t)=&h(x,y,t;z)=-h(x,t,y;z)=-h(x,t,z;y)\\
&=h(x,z,t;y)=h(x,z,y;t)=-h(x,y,z;t)
\end{split}
\]
implies $h=0$. Hence, $(x\ast y)\ast \{z,t\}=0$ and $\{(x\ast y)\ast z,t\}=0$ i.e.\
$A_\ast^2\ast\{A,A\}=0$ and $\{A_\ast^3,A\}=0$.

Since $D$ is a derivation of $(A,\ast)$, by induction one can show that $D(A_\ast^i)\subseteq A_\ast^i$ i.e. $A_\ast^i$ is $D$-stable.
Moreover, $A_\ast^3\ast A=A_\ast^4=0$ and $\{A_\ast^3,A\}=0$, so $A_\ast^3$
is an ideal annihilated by both operations. In particular, $D$ induces a
well-defined map on $A/A_\ast^3$. Since the  term
$(x\ast y)\ast D(z)$ lies in $A_\ast^3$, it vanishes in the quotient. Hence the
induced operations satisfy the compatibility identity in the mock-Poisson algebra.
\end{proof}


\begin{proposition}\label{prop:dimension-seven}
Let $\mathcal A=(A,*,\{-,-\},D)$ be a finite-dimensional generalized mock-Poisson algebra. If $(x\ast y)\ast D(z)\neq0$ for some $x,y,z\in A$, then $\dim A\ge7$.
\end{proposition}

\begin{proof}
The hypothesis implies $A_\ast^3\neq0$, so there are $x,y,z\in A$ with
$f(x,y,z)=(x\ast y)\ast z\neq0$, where $f$ is totally antisymmetric.
Then $x,y,z$ are linearly independent modulo $A_\ast^2$. Because if $x,y,z$ are linearly dependent modulo $A_\ast^2$, we may assume
$z=\alpha x+\beta y+w$ with $w\in A_\ast^2$. Then
\[
f(x,y,z)=\alpha f(x,y,x)+\beta f(x,y,y)+f(x,y,w)=(x\ast y)\ast w\in A_\ast^2\ast A_\ast^2\subseteq A_\ast^4=0,
\]
which is a contradiction. Similarly, we can show that the elements $x\ast y,\ y\ast z,\ z\ast x$ are linearly
independent modulo $A_\ast^3$. Indeed, if
\[
\alpha(x\ast y)+\beta(y\ast z)
+\gamma(z\ast x)\in A_\ast^3.
\]
Multiplying by $z$ and using $A_\ast^4=0$,
we obtain $\alpha(x\ast y)\ast z=0$. Since $(x\ast y)\ast z\neq0$, we get
$\alpha=0$. Multiplying by $x$
and $y$ gives $\beta=\gamma=0$. Hence
\[
\dim A=\dim(A/A_\ast^2)+\dim(A_\ast^2/A_\ast^3)+\dim A_\ast^3\ge3+3+1=7.
\]
\end{proof}

\begin{proposition}
\label{prop:unique-seven-dimensional}
Let $(A,*)$ be a seven-dimensional anti-commutative
anti-associative algebra over a field of characteristic zero
with $A_*^3\neq 0$. Then $A$ is isomorphic to the free
anti-commutative anti-associative algebra $F_3$ on three
generators. In particular, such an algebra is unique up to
isomorphism.
\end{proposition}

\begin{proof}
Since $A_\ast^3\neq0$, we can choose
$x,y,z\in A$ such that $(x\ast y)\ast z\neq0$.
As shown in the proof of
Proposition~\ref{prop:dimension-seven},
the elements $x,y,z$ are linearly independent
modulo $A_\ast^2$, and
$x\ast y$, $y\ast z$, $z\ast x$ are linearly
independent modulo $A_\ast^3$.

Since $\dim A=7$, it follows that
\[
A=\operatorname{span}\{x,y,z\}+A_\ast^2,
\quad
A_\ast^2=
\operatorname{span}\{x\ast y,y\ast z,z\ast x\}
+A_\ast^3,
\quad
A_\ast^3=\operatorname{span}\{(x\ast y)\ast z\}.
\]
Thus, $A$ is generated by $x,y,z$.

By the universal property of $F_3$, there is
a surjective homomorphism $\varphi:F_3\to A$
sending the free generators to $x,y,z$.
Since $\dim F_3=7=\dim A$, the map
$\varphi$ is an isomorphism.
\end{proof}

\begin{corollary}\label{cor:unique-seven-gmp}
Let $\mathcal A=(A,*,\{-,-\},D)$ be a seven-dimensional generalized mock-Poisson algebra over a field of characteristic zero. If $(x\ast y)\ast D(z)\neq0$ for some $x,y,z\in A$, then its underlying anti-commutative anti-associative algebra $(A,\ast)$ is isomorphic to $F_3$.
\end{corollary}
\begin{proof}
The hypothesis implies $A_\ast^3\neq0$. The assertion follows immediately from Proposition~\ref{prop:unique-seven-dimensional}.
\end{proof}

\begin{corollary}\label{cor:low-dim-mock}
Every generalized mock-Poisson algebra of dimension at most six satisfies $(x\ast y)\ast D(z)=0$ for all $x,y,z\in A$ and hence is a mock-Poisson algebra with a common derivation $D$.
\end{corollary}


We now give a construction of generalized mock-Poisson algebras using derivations of anti-commutative anti-associative algebras.

\begin{theorem}\label{thm:derivation-construction}
Let $(A,\ast)$ be an anti-commutative anti-associative algebra and let $D$ be a derivation of $(A,\ast)$. Define a new bilinear operation as
\begin{equation*}
\{x,y\}_D:=D(x)\ast y+ D(y)\ast x.
\end{equation*}
Then $(A,\ast,\{-,-\}_D,D)$ is a generalized mock-Poisson algebra.
\end{theorem}
\begin{proof}
By definition, it is obvious that $\{y,x\}_D=\{x,y\}_D$. The computation 
\[
\begin{split}
D\{x,y\}_D=&D(D(x)\ast y+ D(y)\ast x)\\
=&D^2(x)\ast y+D(x)\ast D(y)+D^2(y)\ast x+D(y)\ast D(x)\\
=&\{D(x),y\}_D+\{x,D(y)\}_D,
\end{split}
\]
implies that $D$ is a derivation of the bracket $\{-,-\}_D$.

Using anti-commutativity, we have
\begin{align*}
\{\{x,y\}_D,z\}_D
={}&(D^2(x)\ast y)\ast z+(D^2(y)\ast x)\ast z\\
&+D(z)\ast (D(x)\ast y)+D(z)\ast (D(y)\ast x).
\end{align*}
By anti-commutativity and anti-associativity, in the cyclic sum, the terms containing $D^2$ cancel pairwise, and the remaining terms cancel pairwise as well. Hence $\{-,-\}_D$ satisfies the Jacobi identity.

Now we check compatibility condition. We have
\begin{align*}
\{x\ast y,z\}_D
&=(D(x)\ast y)\ast z+(x\ast D(y))\ast z+D(z)\ast (x\ast y).
\end{align*}
On the other hand,
\begin{align*}
-x\ast\{y,z\}_D&=(x\ast D(y))\ast z+D(z)\ast(x\ast y),\\
-\{x,z\}_D\ast y&=(D(x)\ast y)\ast z+D(z)\ast(x\ast y).
\end{align*}
Adding these identities and then adding $(x\ast y)\ast D(z)$ yields exactly $\{x\ast y,z\}_D$. This completes the proof.
\end{proof}

\begin{example}
Let $A$ be the three-dimensional anti-associative algebra with non-zero products
\[
e_1\ast e_2=e_3,
\qquad
e_2\ast e_1=-e_3.
\]
The derivation of $A$ is of the form:
\[
D(e_1)=\alpha_1e_1+\beta_1e_2+\gamma_1e_3,\qquad
D(e_2)=\alpha_2e_1+\beta_2e_2+\gamma_2e_3,\qquad
D(e_3)=(\alpha_1+\beta_2)e_3.
\]
Considering the bracket $\{x,y\}_D=D(x)\ast y+D(y)\ast x$, we obtain the following non-zero products
\[
\begin{split}
\{e_1,e_1\}=-2\beta_1e_3,\; \{e_1,e_2\}=(\alpha_1-\beta_2)e_3, \; \{e_2,e_2\}=2\alpha_2e_3.
\end{split}
\]
Consequently, $(A,\ast,\{-,-\}_D,D)$ is a generalized mock-Poisson algebra
and $(A,\ast,\{-,-\}_D)$ is also
a mock-Poisson algebra since $A_*^3=0$.
\end{example}

In the following example we will present a generalized mock-Poisson algebra which is not a mock-Poisson algebra.

\begin{example}\label{ex:nonzero-generalized-correction}
Let $A$ be the seven-dimensional anti-commutative anti-associative algebra with basis $\{e_1,\dots,e_7\}$ described in \cite[Remark~7]{AdashevKaygorodov}. It is free on the generators $e_1,e_2,e_3$, and its nonzero products, up to anti-commutativity, are
\[
e_1\ast e_2=e_4,\quad e_3\ast e_1=e_5,\quad e_2\ast e_3=e_6,\quad
e_4\ast e_3=e_5\ast e_2=e_6\ast e_1=-e_7.
\]
Define the grading derivation
\[
\begin{split}
&D(e_1)=e_1,\quad D(e_2)=e_2,\quad D(e_3)=e_3,\\
&D(e_4)=2e_4,\quad D(e_5)=2e_5,\quad D(e_6)=2e_6,
\quad D(e_7)=3e_7.    
\end{split}
\]
Then $D\in\operatorname{Der}(A,\ast)$. By Theorem~\ref{thm:derivation-construction}, $\{u,v\}_D=D(u)\ast v+D(v)\ast u$ defines a generalized mock-Poisson bracket. Its nonzero values on basis elements, up to symmetry, are
\[
\{e_3,e_4\}_D=-e_7,\qquad
\{e_1,e_6\}_D=-e_7,\qquad
\{e_2,e_5\}_D=-e_7.
\]
Then we have $(e_1\ast e_2)\ast D(e_3)=e_4\ast e_3=-e_7\neq0$. Therefore, the obtained generalized mock-Poisson
algebra is not a mock-Poisson algebra.
\end{example}

This example shows that the bound in
Proposition~\ref{prop:dimension-seven} is sharp. Thus, dimension seven is the smallest dimension in which a generalized mock-Poisson algebra is not necessarily a mock-Poisson algebra.

We now present another construction of generalized mock-Poisson algebras from anti-Zinbiel algebras. Recall that Zinbiel algebras were introduced as the Koszul dual counterpart of Leibniz algebras. Recently, authors of \cite{AdashevKaygorodov} introduced $\delta$-Zinbiel algebras by computing the quadratic dual of the operad governing $\delta$-Leibniz algebras. The specialization $\delta=-1$ gives the notion of an anti-Zinbiel algebra, Koszul dual to an anti-Leibniz algebra (for the definition see \cite{BCKM}). Hence, we have
\begin{definition}[\cite{AdashevKaygorodov}]
An algebra $(A,\circ)$ is called an \emph{anti-Zinbiel algebra} if
\begin{equation}\label{eq:id anti-zinbiel}
x\circ(y\circ z)=-(x\circ y)\circ z+x\circ(z\circ y)
\end{equation}
for all $x,y,z\in A$.
\end{definition}
\eqref{eq:id anti-zinbiel} is equivalent to
\begin{equation}\label{eq:anti-zinbiel-equivalent}
(x\circ y)\circ z=x\circ(z\circ y)-x\circ(y\circ z).
\end{equation}
Interchanging $y$ and $z$ in \eqref{eq:anti-zinbiel-equivalent} yields
\begin{equation}\label{eq:anti-right-commutative}
(x\circ y)\circ z=-(x\circ z)\circ y,
\end{equation}
so every anti-Zinbiel algebra is anti-right-commutative.

\begin{lemma}\label{lem:anti-zinbiel-commutator}
Let $(A,\circ)$ be an anti-Zinbiel algebra. Then the product
\begin{equation}\label{eq:anti-zinbiel-commutator}
x*y:=x\circ y-y\circ x
\end{equation}
is anti-commutative and anti-associative.
\end{lemma}

\begin{proof}
Anti-commutativity is immediate. Using \eqref{eq:anti-zinbiel-commutator}, we have
\[
\begin{aligned}
(x*y)*z
={}&(x\circ y)\circ z-(y\circ x)\circ z-z\circ(x\circ y)+z\circ(y\circ x)\\
\stackrel{\eqref{eq:anti-zinbiel-equivalent}}{=}&x\circ (z\circ y)-x\circ(y\circ z)-(y\circ x)\circ z+(z\circ x)\circ y-z\circ(y\circ x)+z\circ(y\circ x)\\
\stackrel{\eqref{eq:anti-right-commutative}}{=}&x\circ (z\circ y)-x\circ(y\circ z)+(y\circ z)\circ x-(z\circ y)\circ x\\
={}&-x*(y*z).
\end{aligned}
\]
This completes the proof.
\end{proof}

Note that a derivation $D$ of an  anti-Zinbiel algebra $(A,\circ)$ is compatible with this commutator construction i.e.
\begin{equation}\label{eq:derivation with *}
\begin{split}
D(x*y)
&=D(x\circ y-y\circ x)=D(x)\circ y+x\circ D(y)
  -D(y)\circ x-y\circ D(x)\\
&=D(x)*y+x*D(y).
\end{split}    
\end{equation}

\begin{proposition}
Let $(A,\circ)$ be an anti-Zinbiel algebra and let $D$ be a derivation of $(A,\circ)$. Define
\[
x*y=x\circ y-y\circ x, \qquad \{x,y\}_D
=
D(x)*y+D(y)*x.
\]
Then $(A,*,\{-,-\}_D,D)$ is a generalized mock-Poisson algebra.
\end{proposition}

\begin{proof}
By Lemma~\ref{lem:anti-zinbiel-commutator}, $(A,*)$ is an anti-commutative anti-associative algebra. \eqref{eq:derivation with *} shows that $D$ is a derivation of $(A,*)$. The assertion therefore follows directly from Theorem~\ref{thm:derivation-construction}.
\end{proof}

\subsection{A unified tensor product construction}

The tensor product constructions for generalized Poisson and generalized mock-Poisson algebras admit a common sign-parametrized formulation. For this purpose, we set $\varepsilon\in\{1,-1\}$.

We say that a quadruple $(A,\cdot,[-,-],D)$ is a generalized $\varepsilon$-Poisson algebra if
\begin{align*}
xy=\varepsilon\,yx, \quad (xy)z=\varepsilon\,x(yz),\quad  [x,y]=-\varepsilon\,[y,x],
\end{align*}
the bracket satisfies the Jacobi identity
\begin{equation*}
[[x,y],z]+[[y,z],x]+[[z,x],y]=0,
\end{equation*}
$D$ is a derivation of both operations, and
\begin{equation}\label{eq:delta-generalized-poisson}
[xy,z]=\varepsilon\,x[y,z]+\varepsilon\,[x,z]y+(xy)D(z)
\end{equation}
for all $x,y,z\in A$.

Thus, for $\varepsilon=1$, one recovers a generalized Poisson algebra, while for $\varepsilon=-1$, one obtains a generalized mock-Poisson algebra.

Similarly, for $\varepsilon\in\{1,-1\}$ let $(B,\circ)$ be an $\varepsilon$-commutative $\varepsilon$-associative algebra, that is
\begin{equation}\label{eq:epsilon-algebra}
a\circ b=\varepsilon\,b\circ a,
\qquad
(a\circ b)\circ c
=
\varepsilon\,a\circ(b\circ c)
\end{equation}
for all $a,b,c\in B$.

\begin{theorem}\label{thm:unified-tensor-product}
Let $(A,\cdot,[-,-],D)$ be a generalized $\varepsilon_1$-Poisson algebra and let $(B,\circ)$ be an $\varepsilon_2$-commutative $\varepsilon_2$-associative algebra, where $\varepsilon_1,\varepsilon_2\in\{1,-1\}$. Define on $A\otimes B$
\begin{align*}
(x\otimes a)\star(y\otimes b)
&:=(xy)\otimes(a\circ b),\\
[x\otimes a,y\otimes b]_{\otimes}
&:=[x,y]\otimes(a\circ b),\\
\widetilde D(x\otimes a)
&:=D(x)\otimes a.
\end{align*}
Then $(A\otimes B,\star,[-,-]_{\otimes},\widetilde D)$ is a generalized $(\varepsilon_1\varepsilon_2)$-Poisson algebra.
\end{theorem}

\begin{proof}
Let us set $X=x\otimes a,\; Y=y\otimes b,\; 
Z=z\otimes c$. Then, we have
\[
\begin{aligned}
Y\star X
&=(yx)\otimes(b\circ a)=(\varepsilon_1 xy)\otimes(\varepsilon_2 a\circ b)=\varepsilon_1\varepsilon_2\,X\star Y,
\end{aligned}
\]
and hence the tensor product multiplication is $(\varepsilon_1\varepsilon_2)$-commutative.

Similarly, we obtain
\[
\begin{aligned}
(X\star Y)\star Z&=((xy)z)\otimes((a\circ b)\circ c)=\varepsilon_1\varepsilon_2\,
  x(yz)\otimes a\circ(b\circ c)=\varepsilon_1\varepsilon_2\,X\star(Y\star Z).
\end{aligned}
\]
For the bracket, we have
\[
\begin{aligned}
[Y,X]_{\otimes}
&=[y,x]\otimes(b\circ a)=(-\varepsilon_1[x,y])\otimes
  (\varepsilon_2 a\circ b)=-\varepsilon_1\varepsilon_2\,[X,Y]_{\otimes}.
\end{aligned}
\]
Thus the tensor product bracket has precisely the symmetry corresponding to the parameter $\varepsilon_1\varepsilon_2$.

We next verify the Jacobi identity. From \eqref{eq:epsilon-algebra} (for $\varepsilon=\varepsilon_2$) and $\varepsilon_2^2=1$, one obtains
\begin{equation*}
(a\circ b)\circ c
=
(b\circ c)\circ a
=
(c\circ a)\circ b.
\end{equation*}
Consequently,
\[
\begin{aligned}
&[[X,Y]_{\otimes},Z]_{\otimes}
+[[Y,Z]_{\otimes},X]_{\otimes}
+[[Z,X]_{\otimes},Y]_{\otimes}\\
&\quad=
\bigl(
 [[x,y],z]+[[y,z],x]+[[z,x],y]
\bigr)
\otimes((a\circ b)\circ c)
=0.
\end{aligned}
\]

Since $D$ is a derivation of both operations on $A$, it follows immediately that
\[
\widetilde D(X\star Y)=\widetilde D(X)\star Y+X\star\widetilde D(Y)\quad \text{and}\quad \widetilde D[X,Y]_{\otimes}
=
[\widetilde D(X),Y]_{\otimes}
+
[X,\widetilde D(Y)]_{\otimes}.
\]
It remains to verify the compatibility condition. By \eqref{eq:delta-generalized-poisson},
\[
\begin{aligned}
[X\star Y,Z]_{\otimes}
={}&[xy,z]\otimes((a\circ b)\circ c)\\
={}&\bigl(\varepsilon_1 x[y,z]
+\varepsilon_1[x,z]y+(xy)D(z)\bigr)
\otimes((a\circ b)\circ c).
\end{aligned}
\]
By \eqref{eq:epsilon-algebra}, $a\circ(b\circ c)
=
\varepsilon_2\,(a\circ b)\circ c$ and $(a\circ c)\circ b
=
\varepsilon_2\,(a\circ b)\circ c$. Therefore,
\[
X\star[Y,Z]_{\otimes}=\varepsilon_2\,
x[y,z]\otimes((a\circ b)\circ c) \quad \text{and} \quad [X,Z]_{\otimes}\star Y
=\varepsilon_2\,[x,z]y\otimes((a\circ b)\circ c).
\]
Since $\varepsilon_2^2=1$, we obtain
\[
\begin{aligned}
[X\star Y,Z]_{\otimes}
={}&
\varepsilon_1\varepsilon_2\,
X\star[Y,Z]_{\otimes}+\varepsilon_1\varepsilon_2\,
[X,Z]_{\otimes}\star Y
+(X\star Y)\star\widetilde D(Z).
\end{aligned}
\]
This completes the proof.
\end{proof}

Let $(A_i,\ast_i,D_i)$, $i=1,2$, be anti-commutative anti-associative algebras equipped with derivations $D_i$. For $i=1,2$, define
\begin{equation*}
\{x,y\}_i=D_i(x)\ast_i y+D_i(y)\ast_i x.
\end{equation*}
Then $\mathcal A_i=(A_i,\ast_i,\{-,-\}_i,D_i)$, $i=1,2$, are generalized mock-Poisson algebras.

\begin{proposition}
On $A_1\otimes A_2$, define
\begin{align}\label{eq:tensor-special-product}
(x_1\otimes x_2)\cdot(y_1\otimes y_2)
={}&(x_1\ast_1y_1)\otimes(x_2\ast_2y_2),\\
[x_1\otimes x_2,y_1\otimes y_2]
={}&\{x_1,y_1\}_1\otimes(x_2\ast_2y_2)
+(x_1\ast_1y_1)\otimes\{x_2,y_2\}_2\nonumber\\
\widetilde D={}&D_1\otimes\operatorname{id}_{A_2}
+\operatorname{id}_{A_1}\otimes D_2.\nonumber
\end{align}
Then $\bigl(A_1\otimes A_2,\cdot,[-,-],\widetilde D\bigr)$ is a generalized Poisson algebra.
\end{proposition}

\begin{proof}
From \eqref{eq:tensor-special-product}, it is obvious that $(A_1\otimes A_2,\cdot)$ is a commutative associative algebra.

Now we show that $\widetilde D$ is a derivation of  $(A_1\otimes A_2,\cdot)$. For $X=x_1\otimes x_2$ and $Y=y_1\otimes y_2$, we have
\[
\begin{aligned}
\widetilde D(X\cdot Y)
={}&
D_1(x_1\ast_1y_1)\otimes(x_2\ast_2y_2)+
(x_1\ast_1y_1)\otimes D_2(x_2\ast_2y_2)\\
={}&
(D_1(x_1)\ast_1y_1)\otimes(x_2\ast_2y_2)+
(x_1\ast_1D_1(y_1))\otimes(x_2\ast_2y_2)\\
&+
(x_1\ast_1y_1)\otimes(D_2(x_2)\ast_2y_2)+
(x_1\ast_1y_1)\otimes(x_2\ast_2D_2(y_2))\\
={}&
\widetilde D(X)\cdot Y+X\cdot\widetilde D(Y).
\end{aligned}
\]
Let us compute the bracket induced by $\widetilde D$. 
\[
\begin{aligned}
\widetilde D(X)\cdot Y
={}&
(D_1(x_1)\ast_1y_1)\otimes(x_2\ast_2y_2)+
(x_1\ast_1y_1)\otimes(D_2(x_2)\ast_2y_2),
\end{aligned}
\]
and
\[
\begin{aligned}
X\cdot\widetilde D(Y)
={}&
(x_1\ast_1D_1(y_1))\otimes(x_2\ast_2y_2)+
(x_1\ast_1y_1)\otimes(x_2\ast_2D_2(y_2)),
\end{aligned}
\]
derives
\begin{equation*}
\begin{split}
\widetilde D(X)\cdot Y-X\cdot\widetilde D(Y)
={}&
\bigl(
D_1(x_1)\ast_1y_1
-x_1\ast_1D_1(y_1)
\bigr)
\otimes(x_2\ast_2y_2)\\
&+
(x_1\ast_1y_1)\otimes
\bigl(
D_2(x_2)\ast_2y_2
-x_2\ast_2D_2(y_2)
\bigr)\\
={}&
\{x_1,y_1\}_1\otimes(x_2\ast_2y_2)+
(x_1\ast_1y_1)\otimes\{x_2,y_2\}_2\\
={}&[X,Y].
\end{split}
\end{equation*}
The other part of the proof follows from Theorem 3.13 in \cite{CNOR}.
\end{proof}

\subsection{One-dimensional extensions}

If commutative associative part of a generalized Poisson algebra is not unital, then it is possible to extend the algebra to have a unit element (see Proposition~3.9 and Theorem~3.12 in \cite{CNOR}). However, a nonzero anti-commutative algebra cannot possess a two-sided identity element, so the usual unitization procedure is not available in the mock setting. A natural replacement is to adjoin one distinguished element and control its interaction with the original algebra by linear operators.

Let $\mathcal A=(A,\ast,\{-,-\},D)$ be a generalized mock-Poisson algebra and
consider $A^{\sharp}=A\oplus\mathbb C e$. Let
$T,S\colon A\to A$ be linear maps, $a,u\in A$ and $\lambda\in\mathbb C$.
For all $x,y\in A$, we extend the operations of $A$ to $A^{\sharp}$ by setting
\begin{equation}\label{eq:extension-product}
\begin{split}
&x\ast_{\sharp}y=x\ast y, \qquad e\ast_{\sharp} x=T(x),\qquad
x\ast_{\sharp} e=-T(x),\qquad
e\ast_{\sharp} e=0,\\
&\{x,y\}_{\sharp}=\{x,y\}, \qquad \{e,x\}_{\sharp}=\{x,e\}_{\sharp}=S(x),
\qquad \{e,e\}_{\sharp}=a,\\
&D_{\sharp}(x)=D(x),\qquad
D_{\sharp}(e)=\lambda e+u.
\end{split}
\end{equation}

\begin{theorem}\label{thm:one-dimensional-extension}
With the notation above, $\mathcal A^{\sharp}=(A^{\sharp},\ast_{\sharp},\{-,-\}_{\sharp},D_{\sharp})$ is a generalized mock-Poisson algebra if and only if the following identities hold for all $x,y\in A$:
\begin{align}
T(x)\ast y
&=x\ast T(y)
=-T(x\ast y),
\label{eq:ext1}\\
T^2(x)&=0,
\label{eq:ext2}\\
S(\{x,y\})+\{S(x),y\}+\{S(y),x\}
&=0,
\label{eq:ext3}\\
\{a,x\}+2S^2(x)&=0,
\label{eq:ext4}\\
S(a)&=0,
\label{eq:ext5}\\
D(T(x))
&=T(D(x))+\lambda T(x)+u\ast x,
\label{eq:ext6}\\
D(S(x))
&=S(D(x))+\lambda S(x)+\{u,x\},
\label{eq:ext7}\\
D(a)&=2\lambda a+2S(u),
\label{eq:ext8}\\
\{T(x),y\}
&=-T(\{x,y\})-S(y)\ast x
  +T(x)\ast D(y),
\label{eq:ext9}\\
S(x\ast y)
&=-x\ast S(y)-S(x)\ast y
  -\lambda T(x\ast y)+(x\ast y)\ast u,
\label{eq:ext10}\\
S(T(x))
&=-T(S(x))-a\ast x+T(x)\ast u.
\label{eq:ext11}
\end{align}
\end{theorem}

\begin{proof}
Since $\mathcal A$ is already a generalized mock-Poisson algebra, by multilinearity it suffices to verify the defining identities on tuples containing the  element $e$.

Anti-commutativity follows directly from \eqref{eq:extension-product}. By considering the anti-associative product for the triples $\{e,x,y\}$, $\{x,e,y\}$, $\{e,e,x\}$ we obtain \eqref{eq:ext1} and \eqref{eq:ext2}. 

By the definition the bracket $\{-,-\}_{\sharp}$ is commutative. For the triple $\{e,x,y\}$, the Jacobi identity gives
\[
\{\{e,x\},y\}
+\{\{x,y\},e\}
+\{\{y,e\},x\}=0,
\]
and therefore
\[
S(\{x,y\})+\{S(x),y\}+\{S(y),x\}=0.
\]
For $\{e,e,x\}$ it gives $\{a,x\}+2S^2(x)=0$, and for $(e,e,e)$ it gives $S(a)=0$. Hence the extended bracket is a mock-Lie bracket if and only if \eqref{eq:ext3}--\eqref{eq:ext5} hold.

We now require $D_{\sharp}$ to be a derivation of the extended product. Applying the derivation identity to $e\ast_{\sharp} x=T(x)$ gives
\[
\begin{aligned}
D(T(x))
&=D_{\sharp}(e)\ast_{\sharp} x+e\ast_{\sharp} D(x)\\
&=(\lambda e+u)\ast_{\sharp} x+e\ast_{\sharp} D(x)\\
&=\lambda T(x)+u\ast x+T(D(x)),
\end{aligned}
\]
which is \eqref{eq:ext6}.

Similarly, applying the derivation identity for the bracket to $\{e,x\}_{\sharp}=S(x)$ yields
\[
\begin{aligned}
D(S(x))
&=\{D_{\sharp}(e),x\}_{\sharp}+\{e,D(x)\}_{\sharp}=\lambda S(x)+\{u,x\}+S(D(x)),
\end{aligned}
\]
which is \eqref{eq:ext7}. Applying it to $\{e,e\}_{\sharp}=a$ gives $D(a)=2\lambda a+2S(u)$ and hence \eqref{eq:ext8}.

It remains to consider the generalized mock-Poisson compatibility
\begin{equation}\label{eq:extended-gmp}
\{X\ast_{\sharp} Y,Z\}_{\sharp}
=
-X\ast_{\sharp}\{Y,Z\}_{\sharp}
-\{X,Z\}_{\sharp}\ast_{\sharp} Y
+(X\ast_{\sharp} Y)\ast_{\sharp} D_{\sharp}(Z).
\end{equation}

For the triple $\{e,x,y\}$, identity \eqref{eq:extended-gmp} becomes
\[
\{T(x),y\}=-T(\{x,y\})-S(y)\ast x+T(x)\ast D(y),
\]
which is \eqref{eq:ext9}. 

Similarly, the triples $\{x,y,e\}$ and $\{e,x,e\}$ yield
\eqref{eq:ext10} and \eqref{eq:ext11}, respectively, using \eqref{eq:ext2}.

By anti-commutativity of $\ast$ and symmetry of the bracket, the cases
$\{x,e,y\}$ and $\{x,e,e\}$ reduce to those already considered.
The cases $\{e,e,x\}$ and $\{e,e,e\}$ are automatic because $e\ast e=0$. Therefore, $\mathcal A^{\sharp}$ is a generalized mock-Poisson algebra if and only if \eqref{eq:ext1}--\eqref{eq:ext11} hold.
\end{proof}

\begin{corollary}
Let $\mathcal A=(A,\ast,\{-,-\},D)$ be a generalized mock-Poisson algebra. For every $\lambda\in\mathbb C$, the vector space $A^{\sharp}=A\oplus\mathbb C e$ becomes a generalized mock-Poisson algebra by setting
\[
e\ast A=0,\qquad
\{e,A\}=0,\qquad
\{e,e\}=0, \qquad 
D^{\sharp}|_A=D,
\qquad
D^{\sharp}(e)=\lambda e.
\]
\end{corollary}

\begin{proof}
Take $T=S=0$, $a=u=0$ in Theorem~\ref{thm:one-dimensional-extension}.
\end{proof}

\section{Polarization and depolarization}
\label{sec3}
Polarization and depolarization are fundamental tools in the study of nonassociative algebras, providing a natural framework for investigating their structures and the relationships between different algebraic systems. Polarization is a procedure that decomposes a single bilinear multiplication into its symmetric and skew-symmetric parts, thereby associating a two-operation algebra $(A,\bullet,[-,-])$ with a one-operation algebra $(A,\cdot)$. Conversely, depolarization combines two bilinear operations into a single multiplication, allowing the original algebraic structure to be described in terms of one operation. These techniques have been studied in various settings, including transposed Poisson algebras \cite{Djumadildaev1}, Poisson algebras \cite{GozeRemm}, generalized Poisson algebras \cite{Djumadildaev2}, and (transposed) $\delta$-mock-Poisson algebras \cite{RN}.

Recall that an algebra $(A,\circ)$ is called the mock-Poisson-admissible algebra if for
all $x, y, z \in A$ the following polynomial is zero:
\begin{align}
f(x,y,z):={}&3(x\circ y)\circ z-(y\circ x)\circ z
 +(x\circ z)\circ y \notag\\
&-(y\circ z)\circ x+(z\circ x)\circ y+3x\circ (y\circ z).
\label{eq:MP-polynomial}
\end{align}

Let $(A,\circ
,D)$ be an algebra with a derivation $D$.
Define
\[
x\ast y=\frac{x\circ y-y\circ x}{2},
\qquad
\{x,y\}=\frac{x\circ y+y\circ x}{2}.
\]

\begin{definition}
An algebra $(A,\circ,D)$ is called an
\emph{admissible generalized mock-Poisson algebra} if $D$ is a derivation and
\begin{equation}\label{eq:GMP-admissible}
\Psi(x,y,z):=f(x,y,z)
-4(x\ast y)\ast D(z)-2(y\ast z)\ast D(x)=0
\end{equation}
for all $x,y,z\in A$.
\end{definition}

\begin{theorem}\label{thm:GMP-polarization}
Let $(A,\circ,D)$ be an algebra with a derivation $D$.
Then $(A,\circ,D)$ is an admissible generalized
mock-Poisson algebra if and only if
$(A,\ast,\{-,-\},D)$ is a generalized
mock-Poisson algebra.
\end{theorem}

\begin{proof}
Let us set
\begin{align*}
P_1(x,y,z)
&=4\big((x\ast y)\ast z
       +x\ast(y\ast z)\big),\\
P_2(x,y,z)
&=4\big(\{\{x,y\},z\}
       +\{\{y,z\},x\}
       +\{\{z,x\},y\}\big),\\
P_3(x,y,z)
&=4\big(\{x\ast y,z\}
       +x\ast\{y,z\}
       +\{x,z\}\ast y\big).
\end{align*}
For the generalized compatibility condition, we write
\[
P_3^D(x,y,z)
=P_3(x,y,z)-4(x\ast y)\ast D(z).
\]

We use the polynomial identities for the 
mock-Poisson algebra polarization  \cite[Theorem 6.14]{RN}.

In particular,
\begin{align}
2f(x,y,z)=P_3(y,z,x)-2P_3(y,x,z)-2P_1(y,z,x)-2P_1(z,x,y)+P_2(z,x,y).
\label{eq:ordinary-inverse}
\end{align}

Since the product $\ast$ is anticommutative, we have
\begin{align*}
P_3(y,z,x)-2P_3(y,x,z)=&P_3^D(y,z,x)-2P_3^D(y,x,z)\\
&+4(y\ast z)\ast D(x)+8(x\ast y)\ast D(z).
\end{align*}

Therefore, \eqref{eq:ordinary-inverse}
yields
\begin{align}
2\Psi(x,y,z)
={}&P_3^D(y,z,x)-2P_3^D(y,x,z)-2P_1(y,z,x)-2P_1(z,x,y)+P_2(z,x,y).
\label{eq:generalized-inverse}
\end{align}

Thus, if the polarized structure is a generalized
mock-Poisson algebra, then
\[
P_1=P_2=P_3^D=0,
\]
and \eqref{eq:generalized-inverse} implies
$\Psi=0$.

Conversely, suppose that $\Psi=0$. For convenience, we write $g_D(x,y,z):=-4(x\ast y)\ast D(z)-2(y\ast z)\ast D(x)$. Since the product $\ast$ is anticommutative, we have
\begin{align*}
&g_D(x,y,z)+g_D(z,y,x)-g_D(x,z,y)-g_D(z,x,y)=0
\end{align*}
and
\[
\sum_{\sigma\in S_3}g_D^\sigma(x,y,z)=0.
\]
Moreover, a direct computation gives
\begin{align*}
&g_D(x,y,z)-g_D(y,x,z)-g_D(z,y,x)+g_D(x,z,y)-g_D(y,z,x)\\
&\quad+g_D(z,x,y)=-12(x\ast y)\ast D(z).
\end{align*}

Applying the corresponding identities from
\cite[Theorem 6.14]{RN} to $f=\Psi-g_D$, we obtain
\begin{align}
3P_1(x,y,z)
={}&\Psi(x,y,z)+\Psi(z,y,x)-\Psi(x,z,y)-\Psi(z,x,y),
\label{eq:P1-Psi}\\
3P_2(x,y,z)
={}&\sum_{\sigma\in S_3}\Psi^\sigma(x,y,z),
\label{eq:P2-Psi}\\
3P_3^D(x,y,z)
={}&\Psi(x,y,z)-\Psi(y,x,z)-\Psi(z,y,x)+\Psi(x,z,y)-\Psi(y,z,x)+\Psi(z,x,y).
\label{eq:P3-Psi}
\end{align}
Therefore $\Psi=0$ implies $P_1=P_2=P_3^D=0$.

Since $D$ is a derivation of $\circ$,
it is also a derivation of both polarized
operations. Hence
$(A,\ast,\{-,-\},D)$ is a generalized
mock-Poisson algebra.
\end{proof}

\section{Nilpotency of generalized mock-Poisson algebras}\label{sec4}

Nilpotency of anti-associative algebras (\cite[Proposition~2]{R22}) and mock-Lie algebras (\cite[Lemma~2.5]{BF}) leads to study the nilpotency of generalized mock-Poisson algebras. In general, nilpotency of two individual binary operations does not force the vanishing of arbitrarily long products that mix them. However, the compatibility identity \eqref{eq:gmp-nil-compatibility} provides the additional rewriting mechanism needed to control such products.

According to the notion of nilpotency for algebras with two operations given in \cite{ABCF}, we now write it for the generalized mock-Poisson algebras.

\begin{definition}
Let $\mathcal{A}=(A,\ast,\{-,-\},D)$ be a generalized mock-Poisson algebra. The \emph{descending series of powers} is the sequence of subspaces given by
\[
A^{\langle 1\rangle}=A, \quad
A^{\langle n\rangle}
=
\sum_{i=1}^{n-1}
\left(
A^{\langle i\rangle}\ast A^{\langle n-i\rangle}
+
\{A^{\langle i\rangle},A^{\langle n-i\rangle}\}
\right), \quad n\geq2.
\]
The algebra $\mathcal{A}$ is called nilpotent if there exists $N \geq 1$ such that $A^{\langle N\rangle}=0$. 
\end{definition}


By definition $A^{\langle n\rangle}$ is spanned by the values of all mixed non-associative monomials of degree $n$ in the operations $\ast$ and $\{-,-\}$ (the degree of a monomial being the number of its variables). Hence, $\mathcal A$ is nilpotent if and only if all mixed monomials of sufficiently large degree vanish. Similarly one defines the series $A_*^{n}$ and $A_{\{-,-\}}^{n}$ of the two operations separately. Note that the derivation $D$ plays no role in these definitions.

A monomial in which only the operation $\{-,-\}$ occurs will be called a \emph{pure bracket monomial}. A \emph{regular monomial} is a monomial built from pure bracket monomials by means of the operation $\ast$.

\begin{lemma}\label{lem:regular-form}
Let $\mathcal A=(A,\ast,\{-,-\},D)$ be a generalized mock-Poisson algebra. Then every mixed monomial of degree $n$ is a linear combination of regular monomials of degree $n$.
\end{lemma}

\begin{proof}
We argue by induction on the degree. For a monomial $M_1\ast M_2$ it is clear since by induction $M_1,M_2$ are combinations of regular monomials and $\ast$-products of regular monomials are regular. Now for a monomial $\{M_1,M_2\}$ it remains to consider $\{R_1,R_2\}$, where $R_1$ and $R_2$ are regular monomials. Let $c(R)$ be the number of $\ast$-operations occurring in $R$ outside of its pure bracket factors, and use induction on $c(R_1)+c(R_2)$. If $c(R_1)=c(R_2)=0$, then $R_1,R_2$ are pure bracket monomials, so $\{R_1,R_2\}$ is one as well. Otherwise, by symmetry of $\{-,-\}$ we may assume $R_1=S_1\ast S_2$ with regular $S_1,S_2$. By the compatibility condition \eqref{eq:gmp-nil-compatibility},
\[
\{S_1\ast S_2,R_2\}=-S_1\ast\{S_2,R_2\}-\{S_1,R_2\}\ast S_2+(S_1\ast S_2)\ast D(R_2).
\]
Since $c(S_2)+c(R_2)<c(R_1)+c(R_2)$ and $c(S_1)+c(R_2)<c(R_1)+c(R_2)$, the monomials $\{S_2,R_2\}$ and $\{S_1,R_2\}$ are combinations of regular monomials by the induction hypothesis, and so are their $\ast$-products with $S_1$, $S_2$. For the last term, since $D$ is a derivation of both operations $D(R_2)$ is a sum of monomials of the same shape as $R_2$. These are regular monomials of the same degree as $R_2$, and therefore $(S_1\ast S_2)\ast D(R_2)$ is a combination of regular monomials. All monomials obtained have degree $\deg R_1+\deg R_2$.
\end{proof}

\begin{theorem}\label{thm:gmp-nilpotent}
Every finite-dimensional generalized mock-Poisson algebra over a field of characteristic zero is nilpotent. More precisely, if every pure bracket
monomial of degree at least $s$ vanishes, then $A^{\langle 3s-2\rangle}=0$.
\end{theorem}

\begin{proof}
Let $\mathcal A=(A,\ast,\{-,-\},D)$ be a finite-dimensional generalized
mock-Poisson algebra. The mock-Lie algebra $(A,\{-,-\})$ is nilpotent by \cite[Lemma~2.5]{BF}, so there is $s\ge2$ such that every pure bracket monomial of degree $\ge s$ vanishes. Moreover, $A_{\ast}^{4}=0$ by \cite{R22}. 
By anti-associativity, every $\ast$-monomial
is equal to a left-normed one up to a sign.
Since $A_\ast^4=0$, every $\ast$-monomial
of degree at least four vanishes.

By Lemma~\ref{lem:regular-form}, every mixed monomial of degree $n$ is a linear combination of regular monomials of degree $n$. Let $N$ be a nonzero regular monomial of degree $n$, and denote its pure bracket factors by $M_1,\ldots,M_r$. Since $\ast$-monomials of degree $\ge4$ vanish, we necessarily
have $r\leq3$. On the other hand, every nonzero pure bracket factor satisfies
$\deg M_i\leq s-1$. Therefore,
\[
n=\deg N
=\sum_{i=1}^{r}\deg M_i
\leq r(s-1)
\leq3(s-1).
\]
Thus every regular monomial of degree greater than $3(s-1)$ vanishes, and hence so does every mixed monomial of degree greater than $3(s-1)$. Thus, $\mathcal A$ is nilpotent and its nilpotency index is at most $3(s-1)+1$.
\end{proof}

\begin{corollary}\label{cor:mp-nilpotent}
Every finite-dimensional mock-Poisson algebra $(A,\ast,\{-,-\})$ over a field of characteristic zero is nilpotent.
\end{corollary}

\begin{proof}
This follows from Theorem~\ref{thm:gmp-nilpotent}
by taking $D=0$.
\end{proof}

\begin{remark}
Note that no nilpotency assumption on the derivation $D$ is required. Thus, in contrast with  generalized Poisson algebras, an Engel-type hypothesis is not needed in finite-dimensional generalized mock-Poisson algebras because nilpotency is automatic.
\end{remark}

\begin{corollary}\label{cor:flag}
Let $\mathcal A=(A,*,\{-,-\},D)$ be a finite-dimensional generalized mock-Poisson algebra over $\mathbb C$. There exists a $D$-stable flag
\[
0=A_0\subset A_1\subset\cdots\subset A_n=A,
\qquad \dim A_i=i,
\]
such that
\[
A\ast A_i+\{A,A_i\}\subseteq A_{i-1}
\qquad (1\le i\le n).
\]
Consequently, in a basis adapted to this flag all left multiplication operators for $*$ and all multiplication operators for $\{-,-\}$ are strictly upper triangular, while $D$ is upper triangular.
\end{corollary}

\begin{proof}
We argue by induction on $n=\dim A$. The case
$n=0$ is trivial. Suppose $n\ge1$.
By Theorem~\ref{thm:gmp-nilpotent}, $\mathcal A$
is nilpotent. Hence the last nonzero term
of its descending series is contained in
$\operatorname{Ann}(\mathcal A)$, so
$\operatorname{Ann}(\mathcal A)\neq0$.

Since $D$ is a derivation of both operations,
$\operatorname{Ann}(\mathcal A)$ is $D$-stable.
As $\mathbb C$ is algebraically closed,
$D|_{\operatorname{Ann}(\mathcal A)}$ has an
eigenvector $v\neq0$.
Set $A_1=\mathbb C v$. Then $A_1$ is a
$D$-stable ideal and
\[
A\ast A_1=\{A,A_1\}=0.
\]

The quotient $\mathcal A/A_1$ is a
generalized mock-Poisson algebra of
dimension $n-1$. By induction, it admits
a $D$-stable flag with the required properties.
Taking the preimages of its terms and
preceding them by $A_1$ gives the desired
flag of $A$.
\end{proof}

\begin{corollary}\label{cor:no-simple}
Every finite-dimensional generalized mock-Poisson
algebra over $\mathbb C$ of dimension greater
than one has a nonzero proper $D$-stable ideal
annihilated by both operations.
\end{corollary}

\begin{proof}
The subspace $A_1$ from
Corollary~\ref{cor:flag} is such an ideal.
\end{proof}

\section{Representations of mock-Poisson and generalized mock-Poisson algebras}\label{sec5}

Representations of mock-Lie algebras have been studied in several papers, and semidirect products give a natural way to describe them \cite{AgoreMilitaru,BenaliChtiouiHajjajiMabrouk}. Since no systematic representation theory for mock-Poisson algebras appears to be available, we begin this section by studying their representations through the semidirect-product construction. We then extend the discussion to generalized mock-Poisson algebras.

\subsection{Representations of mock-Poisson algebras}

Let $\mathcal A=(A,\ast,\{-,-\})$ be a mock-Poisson algebra, $V$ a vector space, and $L,\rho\colon A\longrightarrow\operatorname{End}(V)$ linear maps. We write $L_x=L(x),\; \rho_x=\rho(x)$.

\begin{definition}\label{def:MP-representation}
A \emph{representation} of the mock-Poisson algebra $\mathcal A$ on $V$ is a triple $(V,L,\rho)$ such that
\begin{align}
L_{x\ast y}&=-L_xL_y,\label{eq:MP-rep1}\\
\rho_{\{x,y\}}&=-\rho_x\rho_y-\rho_y\rho_x,\label{eq:MP-rep2}\\
\rho_{x\ast y}
&=-L_x\rho_y+L_y\rho_x,
\label{eq:MP-rep3}\\
\rho_zL_x
&=-L_x\rho_z-L_{\{x,z\}}
\label{eq:MP-rep4}
\end{align}
for all $x,y,z\in A$. Thus \eqref{eq:MP-rep1} says that $L$ is a representation of $(A,\ast)$, and \eqref{eq:MP-rep2} says that $\rho$ is a representation of the mock-Lie algebra $(A,\{-,-\})$. A subspace $W\subseteq V$ is a \emph{subrepresentation} if $L_xW+\rho_xW\subseteq W$ for all $x\in A$. The representation is \emph{irreducible} if $V\neq0$ and it has no subrepresentations other than $0$ and $V$.
\end{definition}

Note that \eqref{eq:MP-rep1} and anti-commutativity of $\ast$ imply $L_xL_y+L_yL_x=0$.

\begin{theorem}\label{thm:MP-semidirect}
Let $\mathcal A=(A,\ast,\{-,-\})$ be a mock-Poisson algebra, $V$ a vector space, and $L,\rho\colon A\longrightarrow\operatorname{End}(V)$ linear maps. Define two bilinear operations on $A\oplus V$ by
\begin{align*}
(x,u)\widetilde\ast(y,v)
&=
\bigl(x\ast y,L_xv-L_yu\bigr),
\nonumber\\
\{(x,u),(y,v)\}_{\sim}
&=
\bigl(\{x,y\},\rho_xv+\rho_yu\bigr).
\end{align*}
Then $(A\oplus V,\widetilde\ast,\{-,-\}_{\sim})$ is a mock-Poisson algebra if and only if $(V,L,\rho)$ is a representation of $\mathcal A$.
\end{theorem}

\begin{proof}
Let us set
\[
X=(x,u),\qquad Y=(y,v),\qquad Z=(z,w).
\]
Clearly $\widetilde\ast$ is anti-commutative and $\{-,-\}_\sim$ is symmetric. 

\emph{Anti-associativity.} The $V$-component of $(X\widetilde\ast Y)\widetilde\ast Z+X\widetilde\ast(Y\widetilde\ast Z)$ equals
\[
(L_{x\ast y}+L_xL_y)w-(L_zL_x+L_xL_z)v+(L_zL_y-L_{y\ast z})u .
\]
If \eqref{eq:MP-rep1} holds, then all three coefficients vanish since $L_xL_z+L_zL_x=0$ and $L_{y\ast z}=-L_yL_z=L_zL_y$. Conversely, vanishing of the coefficient of $w$ gives \eqref{eq:MP-rep1}.

\emph{Jacobi identity.} The $V$-component of $\{\{X,Y\}_\sim,Z\}_\sim+\{\{Y,Z\}_\sim,X\}_\sim+\{\{Z,X\}_\sim,Y\}_\sim$ has the coefficient $\rho_{\{x,y\}}+\rho_x\rho_y+\rho_y\rho_x$ at $w$. The coefficients at $u$ and $v$ are obtained from it by cyclic permutations of the variables. Hence the Jacobi identity holds if and only if \eqref{eq:MP-rep2} holds.

\emph{Compatibility.} The $V$-component of $\{X\widetilde\ast Y,Z\}_{\sim}$ is
\[
\rho_{x\ast y}w+\rho_zL_xv-\rho_zL_yu.
\]
On the other hand, the $V$-component of $-X\widetilde\ast\{Y,Z\}_{\sim}
-\{X,Z\}_{\sim}\widetilde\ast Y$ is
\[
(-L_x\rho_y+L_y\rho_x)w
+
(-L_x\rho_z-L_{\{x,z\}})v
+
(L_y\rho_z+L_{\{y,z\}})u.
\]
Hence comparison of the coefficients of $w$ and $v$ gives \eqref{eq:MP-rep3} and \eqref{eq:MP-rep4}, respectively. The coefficient of $u$ gives the same identity \eqref{eq:MP-rep4} with $x$ replaced by $y$.

Thus $A\oplus V$ is a mock-Poisson algebra if and only if \eqref{eq:MP-rep1}--\eqref{eq:MP-rep4} hold.
\end{proof}

The resulting mock-Poisson algebra is denoted by $A\ltimes_{(L,\rho)}V$ and is called the \emph{semidirect product} associated to the representation $(V,L,\rho)$.

\begin{example}
Let $(A,\ast,\{-,-\})$ be a mock-Poisson algebra. Define $\mathrm{ad_1}, \mathrm{ad_2} :A \rightarrow  \mathfrak{gl}(A)$ by
\[
{\mathrm{ad}_1}_x(y)=x\ast y,
\qquad
{\mathrm{ad}_2}_x(y)=\{x,y\}.
\]
Then $(A,\mathrm{ad}_1,\mathrm{ad}_2)$ is a representation of $A$, which is called the adjoint representation.
\end{example}

\subsection{Representations of generalized mock-Poisson algebras}

We now turn to representations of generalized mock-Poisson algebras. Let 
$\mathcal{A}=(A,\ast,\{-,-\},D)$ be such an algebra. Besides the two linear maps $L,\rho\colon A\longrightarrow\operatorname{End}(V)$, the presence of $D$ forces us to introduce one more linear operator $d\in\operatorname{End}(V)$, which plays the role of the derivation. 
\begin{definition}\label{def:GMP-representation}
Let $(A,\ast,\{-,-\},D)$ be a generalized mock-Poisson algebra. A \emph{representation} of $\mathcal A$ is a quadruple $(V,L,\rho,d)$ such that \eqref{eq:MP-rep1}, \eqref{eq:MP-rep2} hold (i.e., $L$ is a representation of $(A,\ast)$ and $\rho$ is a representation of $(A,\{-,-\})$) and
\begin{align}
dL_x
&=L_{D(x)}+L_xd,
\label{eq:GMP-rep3}\\
d\rho_x
&=\rho_{D(x)}+\rho_xd,
\label{eq:GMP-rep4}\\
\rho_{x\ast y}
&=-L_x\rho_y+L_y\rho_x+L_{x\ast y}d,
\label{eq:GMP-rep5}\\
\rho_zL_x
&=-L_x\rho_z-L_{\{x,z\}}-L_{D(z)}L_x
\label{eq:GMP-rep6}
\end{align}
for all $x,y,z\in A$. Subrepresentations are $d$-stable subspaces $W\subseteq V$ with $L_xW+\rho_xW\subseteq W$ and irreducibility is defined as above.
\end{definition}

\begin{theorem}\label{thm:GMP-semidirect}
Let $(A,\ast,\{-,-\},D)$ be a generalized mock-Poisson algebra. For a vector space $V$, linear maps $L,\rho\colon A\longrightarrow\operatorname{End}(V)$ and $d\in\operatorname{End}(V)$, define
\begin{align*}
(x,u)\widetilde\ast(y,v)
&=
\bigl(x\ast y,L_xv-L_yu\bigr),\\
\{(x,u),(y,v)\}_{\sim}
&=
\bigl(\{x,y\},\rho_xv+\rho_yu\bigr),\\
\widetilde D(x,u)
&=
\bigl(D(x),d(u)\bigr).
\end{align*}
Then $(A\oplus V,\widetilde\ast,\{-,-\}_{\sim},\widetilde D)$ is a generalized mock-Poisson algebra if and only if $(V,L,\rho,d)$ is a representation of $(A,\ast,\{-,-\},D)$.
\end{theorem}

\begin{proof}
As in the proof of Theorem~\ref{thm:MP-semidirect}, $\widetilde\ast$ is anti-associative and $\{-,-\}_\sim$ satisfies the Jacobi identity if and only if \eqref{eq:MP-rep1} and \eqref{eq:MP-rep2} hold.

The condition that $\widetilde D$ be a derivation of $\widetilde\ast$ and of $\{-,-\}_\sim$ is equivalent to \eqref{eq:GMP-rep3} and \eqref{eq:GMP-rep4} by comparing the $V$-components.

It remains to verify the generalized mock-Poisson identity. 
\[
\{X\widetilde\ast Y,Z\}_{\sim}
=
-X\widetilde\ast\{Y,Z\}_{\sim}
-\{X,Z\}_{\sim}\widetilde\ast Y
+(X\widetilde\ast Y)\widetilde\ast\widetilde D(Z).
\]
Let us set $X=(x,u),\; Y=(y,v),\; Z=(z,w)$. The $V$-component of the last term is $L_{x\ast y}d\,w-L_{D(z)}(L_xv-L_yu)$. By comparing the coefficient of $w$ and $v$ in the $V$-component, we obtain \eqref{eq:GMP-rep5} and \eqref{eq:GMP-rep6}, respectively. The coefficient of $u$ gives the same identity \eqref{eq:GMP-rep6}  with $x$ replaced by $y$. This
completes the proof.
\end{proof}

We denote the resulting generalized mock-Poisson algebra by $A\ltimes_{(L,\rho,d)}V$.

\begin{corollary}\label{cor:irreducible-representations}
Let $\mathcal A=(A,\ast,\{-,-\},D)$ be a finite-dimensional generalized
mock-Poisson algebra over $\mathbb C$ and let $(V,L,\rho,d)$ be a nonzero
finite-dimensional irreducible representation of $\mathcal A$. Then
\[
\dim V=1,\qquad L=0,\qquad \rho=0,
\]
and $d$ acts by a scalar.
\end{corollary}

\begin{proof}
By Theorem~\ref{thm:GMP-semidirect}, $A\ltimes_{(L,\rho,d)}V$ is a
finite-dimensional generalized mock-Poisson algebra, hence nilpotent with
respect to mixed products by Theorem~\ref{thm:gmp-nilpotent}. Put
$V^{(0)}=V$ and $V^{(k+1)}=L_AV^{(k)}+\rho_AV^{(k)}$. Since $V$ is an abelian
ideal, one sees by induction that $V^{(k)}$ is contained in the $(k+1)$-st term of the descending series of the semidirect product, so nilpotency gives $V^{(N)}=0$ for some $N$. Taking the last nonzero
$V^{(k)}$, we get
\[
V_0:=\{v\in V:\ L_xv=0=\rho_xv\ \text{for all }x\in A\}\supseteq V^{(k)}\neq0.
\]
By \eqref{eq:GMP-rep3}--\eqref{eq:GMP-rep4}, $V_0$ is $d$-stable (if $v\in V_0$, then $L_xdv=dL_xv-L_{D(x)}v=0$ and similarly for $\rho_x$), so $V_0$ is a
nonzero subrepresentation. Irreducibility gives $V_0=V$, i.e.\ $L=\rho=0$.
Then every $d$-stable subspace of $V$ is a subrepresentation. Since
$\mathbb C$ is algebraically closed, $d$ has an eigenvector $v$.
Thus $V=\mathbb Cv$ and $d$ acts by a scalar.
\end{proof}

\begin{example}
Let $\mathcal A=(A,\ast,\{-,-\},D)$ be a generalized mock-Poisson algebra. Define $\mathrm{ad_1}, \mathrm{ad_2} :A \rightarrow  \mathfrak{gl}(A)$ by
\[
{\mathrm{ad}_1}_x(y)=x\ast y,
\qquad
{\mathrm{ad}_2}_x(y)=\{x,y\},
\]
and set $d=D$. Then $(A,\mathrm{ad}_1,\mathrm{ad}_2,d)$ is a representation of $A$, which we call the adjoint representation. Indeed, \eqref{eq:MP-rep1} and \eqref{eq:MP-rep2} are anti-associativity and the Jacobi identity, \eqref{eq:GMP-rep3} and \eqref{eq:GMP-rep4} express that $D$ is a derivation, and \eqref{eq:GMP-rep5}, \eqref{eq:GMP-rep6} are the compatibility condition \eqref{eq:gmp-nil-compatibility} evaluated at $(x,y,z)$ and at $(x,w,z)$, respectively.
\end{example}

Note that if $D=0$ and $d=0$, then Definition~\ref{def:GMP-representation} reduces exactly to Definition~\ref{def:MP-representation}. Thus, representations of generalized mock-Poisson algebras form a natural extension of the representation theory introduced above for mock-Poisson algebras.

\end{document}